\documentclass[11pt]{extarticle}
\usepackage[totalwidth=480pt,totalheight=680pt]{geometry}
\usepackage{amsmath}
\usepackage[htt]{hyphenat}
\usepackage{amsfonts}
\usepackage{amssymb}
\usepackage{amsthm}
\usepackage{bm}
\usepackage{tikz-cd}
\usepackage[prependcaption]{todonotes}

\usepackage{times}
\usepackage{microtype}

\definecolor{purple1}{HTML}{5903AA}
\definecolor{blue1}{HTML}{220D82}
\definecolor{blue2}{HTML}{078EBF}
\definecolor{red1}{HTML}{600909}
\usepackage{hyperref}
\hypersetup{
  colorlinks=true,
  linkcolor=purple1,
  citecolor=blue2,
  filecolor=red1,
  urlcolor=blue1
}
\usepackage[capitalize]{cleveref}

\crefname{equation}{eq.}{equations}  
\crefname{section}{sec.}{sections}
\crefname{lemma}{lem.}{lemmata} 
\crefname{proposition}{prop.}{propositions} 
\crefname{remark}{rem.}{remarks}
\crefname{enumi}{step}{steps}
\crefname{theorem}{thm.}{theorems}
\crefname{algorithm}{alg.}{algorithms}
\crefname{figure}{fig.}{figures}
\crefname{definition}{def.}{definitions}
\crefname{observation}{obs.}{observations}
\crefname{corollary}{cor.}{cor.}
\crefname{appendix}{app.}{app.}

\usepackage{xspace}

\usepackage[linesnumbered,ruled,vlined,boxruled,boxed,algo2e]{algorithm2e}
\usepackage[noend]{algorithmic}

\usepackage{et-notations}
\usepackage{nicematrix}

\newtheorem{theorem}{Theorem}

\newtheorem{definition}[theorem]{Definition}
\newtheorem{proposition}[theorem]{Proposition}
\newtheorem{corollary}[theorem]{Corollary}
\newtheorem{remark}[theorem]{Remark}

\newtheorem{example}[theorem]{Example}

\def\cW{{Y}\xspace}
\def\cV{{X}\xspace}
\def\cU{\mathcal{U}\xspace}

\def\VV{\mathbb{V}\xspace}
\newcommand{\pp}{\mathbb{P}}
\newcommand{\PP}{\mathbb{P}}

\newcommand{\pdg}{\mathfrak{d}\xspace}

\begin{document}

\title{Segre-Driven Radicality Testing}

\author{
  Martin Helmer%
  \thanks{Mathematical Sciences Institute, The Australian National University, Canberra, Australia.\newline \texttt{email::martin.helmer@anu.edu.au}}
  \and
  Elias Tsigaridas%
  \thanks{Inria Paris, Institut de Mathématiques de Jussieu - Paris
    Rive Gauche, Sorbonne Universit\'e and Paris Universit\'e, France.
    \texttt{email::elias.tsigaridas@inria.fr}}
}

\maketitle

\begin{abstract}
We present a probabilistic algorithm to test if a homogeneous polynomial ideal
$I$ defining a scheme $X$ in $\PP^n$ is radical using Segre classes and other geometric notions from
intersection theory. Its worst case complexity depends on the geometry of $X$.
If the scheme $X$ has reduced isolated primary components and no embedded components
supported the singular locus of $X_{\rm red}=\VV(\sqrt{I})$, then the worst case
complexity is doubly exponential in $n$; in all the other cases the complexity is singly
exponential. The realm of the ideals for which our radical testing procedure requires only
single exponential time includes examples which are often considered
pathological, such as the ones drawn from the famous Mayr-Meyer set of ideals which
exhibit doubly exponential complexity for the ideal membership problem.


\end{abstract}

\section{Introduction}
\label{sec:Intro}

We consider the problem of testing if an ideal is radical, or in other words if
the associated scheme is reduced. More precisely, for a homogeneous ideal
$I=\langle f_1,\dots, f_r \rangle$ in $\CC[x_0, x_1,\dots, x_n]$ with $d$ being the
maximum degree of the polynomials $f_i$, we present a (probabilistic) algorithm
to test if the scheme $X\subset \PP^n$  defined by $I$ is reduced, i.e.,~to test if $I$ is radical (up to saturation by the irrelevant ideal). When $X$ has
reduced isolated primary components that in addition have no embedded components outside
of the singular locus of $X_{\rm red}=\VV(\sqrt{I})$, then the radical is computed
via a single ideal saturation and the worst case complexity becomes doubly
exponential in $n$. In all the other cases the radical is not explicitly
computed and the worst case complexity is singly exponential in $n$. To
understand what types of embedded components we can deal with, while maintaining
the singly exponential complexity bound, consider the following example.

\begin{example}[Embedded components outside the singular locus of the radical]
  We work in $\PP^2$ with coordinates $x,y,z$ and we consider the scheme $X$
  defined by the
  ideal $$I=\left\langle -x^{2}y^{2}+y^{3}z,-x^{4}+x^{2}y\,z\right\rangle=\langle x^2-yz \rangle \cap \langle y^2,x^4-x^2yz \rangle .$$
  In this case the only isolated primary component of $X$ is the reduced component
  $X_{\rm red}=\mathbb{V}(x^2-yz)$; also $X_{\rm red}$ is a smooth curve in
  $\pp^2$ and so its singular locus is empty. However, $X$ has an embedded
  component $\mathbb{V}(y^2,x^4-x^2y z)$, supported on the point $[0:0:1]$. This
  embedded component would be detected by our algorithm using methods which have
  singly exponential worst case complexity in the number of variables.
\end{example}
 A more interesting example of an ideal $I$ defining a scheme $X$ with embedded components outside the singular locus of $X_{\rm red}$ is furnished by the homogeneous version of the Mayr-Meyer ideals \cite{mayr1982complexity} introduced by Bayer and Stillman \cite{bayer1988complexity}. The Mayr-Meyer family of ideals are generated by polynomials of degree $\mathcal{O}(d)$ in $\mathcal{O}(n)$ variables. Ideals in this family have the property that for some polynomial $f\in I$ the polynomials $r_i$ which solve the ideal membership problem via the expression $f=\sum_i r_if_i$ are such that $\deg(r_i)$ is doubly exponential in $n$, i.e.,~$\mathcal{O}(d^{2^n})$.  In \cite[\S2]{bayer1988complexity} a family of Mayr-Meyer ideals $J_n$ is described in a ring with $10n$ variables and these generators are homogenized to give a homogeneous ideal $J^\prime_n$ in $10n+1$ variables.
 Consider the $n=2$ case; in this case the homogeneous ideals $J^\prime_2$ are ideals in a ring with the $21$ variables $S_0, S_1, F_0, F_1, a_0, a_1, b_0, b_1, c_0, c_1, d_0, d_1, e_0, e_1, f_0, f_1, g_0, g_1, h_0, h_1, z$. One of these ideals is as follows (additional examples may be generated using the Macaulay2 \cite{M2} function \texttt{mayr} in \cite{Mayr_Meyer_M2}):
 {\small
\begin{multline}
    I= \langle {S}_{0}{h}_{0}-{F}_{1}z,{S}_{0}{g}_{0}-{F}_{0}{h}_{0},{S}_{0}{f}_{0}-{F}_{0}{h}_{0},{F}_{0}{e}_{0}-{F
      }_{0}{h}_{0},{S}_{0}{e}_{0}-{S}_{1}z,{F}_{0}{d}_{0}{h}_{0}-{F}_{1}z^{2},{F}_{0}{a}_{0}{h}_{0}-{S}_{1}z^{2}, \\ {F}_{0}{c}_{0}{g
      }_{0}-{F}_{0}{h}_{0}z,{F}_{0}{b}_{0}{f}_{0}-{F}_{0}{h}_{0}z,{F}_{0}{a}_{0}{f}_{0}-{F}_{0}{d}_{0}{g}_{0},{F}_{0}{c}_{0}{d}_{1
      }{f}_{0}{h}_{1}-{F}_{0}{d}_{0}{g}_{0}{h}_{1}z, \\ {F}_{0}{c}_{0}{c}_{1}{f}_{0}{g}_{1}-{F}_{0}{d}_{0}{g}_{0}{g}_{1}z,{F}_{0}{b}_{1
      }{c}_{0}{f}_{0}{f}_{1}-{F}_{0}{d}_{0}{f}_{1}{g}_{0}z,{F}_{0}{a}_{1}{c}_{0}{e}_{1}{f}_{0}-{F}_{0}{d}_{0}{e}_{1}{g}_{0}z \rangle.
    \end{multline}
  }
  
The ideal $I$ defines a scheme $X\subset \PP^{20}$ of codimension three. The
irreducible component $\mathbb{V}(F_0,S_0,z)$ contains the embedded components 
$\mathbb{V}(z,{g}_{0},{f}_{0},{F}_{0},{S}_{0},{F}_{1}{e}_{0}-{S}_{1}{h}_{0})$
and
$\mathbb{V}(z,{f}_{0}-{g}_{0},{e}_{0}-{h}_{0},{d}_{0},{c}_{0},{b}_{0},{a}_{0},{F}_{0},{S}_{1}-{F}_{1},{S}_{0})$;
neither of these components is contained in the singular locus of
$X_{\rm red}=\mathbb{V}(\sqrt{I})$. While it seems to us interesting that many
pathological examples, such as the Mayr-Meyer examples, with potentially
numerous and complicated embedded components, can be considered by our algorithm
in time at most singly exponential in the number of variables, it is worth
noticing that many desirable (and perhaps more mundane seeming) examples, such
as all ideals which are radical, require us to perform an operation (namely a
single ideal saturation) which has a doubly exponential worst case bound. The
latter case of course includes an ideal generated by generic polynomials, which
is expected to be radical, so this is in that sense the "most common" case.
There does, however, seem to be some hope that this saturation operation
could be avoided in some way, as it in fact computes much more information than
we strictly need for  our test.
Nevertheless, if we have the information that the ideal is "generic", then we can combine our approach with other algorithms that have single exponential complexity bounds in this case, e.g., \cite{krick1991membership,armendariz1995computation}.

We now give a brief conceptual overview of our approach. For a homogeneous ideal
$I$ in $\CC[x_0,\dots , x_n]$ we consider the scheme
$X=\mathbb{V}(I)\subset \pp^n$ associated to it. We first sample (at least one)
 generic point $p_i$ in each isolated primary component $W_i$ of $X$ and compute the
multiplicity of $p_i$ in $W_i$ via a calculation which requires only the computation of the
degree of an ideal (we do this in singly exponential time). If a generic point
has multiplicity greater than one, then the associated component is not reduced
and $I$ is not radical. If all isolated primary components are reduced, then we compute
the singularity subscheme of $X$, $\mathbb{S}\mathfrak{ing}(X)=\mathbb{V}(J)$
which is defined by an ideal $J$ whose generators consist of the
${\rm codim}(X)\times {\rm codim}(X)$ minors of the Jacobian matrix of a
generating set of $I$. This scheme $\mathbb{S}\mathfrak{ing}(X)$ has primary
components whose support is either an embedded component of $X$ not
contained in the singular locus of $X_{\rm red}$ or else are supported on the singular
locus of $X_{\rm red}$. We use a multiplicity based test to identify the
presence of embedded components of the first type. If no non-reduced structures
in $X$ have been identified at this stage of the algorithm we see that
$\sqrt{I}=I:J^\infty$; this computation is the {\bf only} computation involved
which has a doubly exponential worst case bound. At this point one could test if
$\sqrt{I}$ equals $I$ using standard methods. Instead, using the computed
generators of $\sqrt{I}$ we compute the {\em Segre class}
$s(X_{\rm red}, \pp^n)$ and using the generators of $I$ we compute the Segre
class $s(X, \pp^n)$, both these operations require  singly exponential time.
Then, using a result we show in \S\ref{subsec:Segre_IntegralClosure}, we
can conclude that $X$ is a reduced scheme in $\pp^n$ if and only if these two
Segre classes agree (in terms of ideals this translates to
$I:\langle x_0, \dots, x_n \rangle^\infty$ being radical).
We should also mention that we present a method to compute the multiplicity of
an irreducible variety contained in some lower dimensional isolated primary
component of a scheme (which is new); see Sections~\ref{sec:mult-subvariety} and
\ref{subsec:affineMult} and also \eqref{eq:MultInComp}. This method requires
only the computation of dimension and the degree of a polynomial ideal and makes
no assumptions on their structure.
We have trialed our algorithms on several test cases using Macaulay 2 \cite{M2}
where we also exploit the various routines for Segre class computations based on \cite{HH2019segre}.

Finally we remark that, while we have not been able to construct a method to
compute the Segre class $s(X_{\rm red}, \pp^n)=s(\mathbb{V}(I:J^\infty), \pp^n)$
without the generators of the ideal $\sqrt{I}=I:J^\infty$ in general; this question
seems worth further consideration. For this reason, we employ the Segre class
based test in the final step of our algorithm rather than some more standard
method to check equality of ideals. In particular, given a singly exponential
algorithm to compute $s(\mathbb{V}(I:J^\infty), \pp^n)$ using only the
generators of $I$ and $J$, our test for $X$ being reduced becomes singly
exponential in all cases. We note that, for example, if $I:J^\infty$ is a
complete intersection, given the degrees of a set of generators, the computation
of $s(\mathbb{V}(I:J^\infty), \pp^n)$ in fact requires only constant time. Such
facts, along with the fact that Segre class computations tend to require only
the computation of degrees of ideals, give us hope that a singly exponential
bound for the computation of $s(\mathbb{V}(I:J^\infty), \pp^n)$ may be possible
using only the generators of $I$ and $J$.

Historically, many different algorithms to test if an ideal $I$ is radical have
been presented in the literature; in most of the cases they test for radicality
in the process of computing the radical or the primary decomposition.
Roughly speaking, these algorithms tend to consist of a step which computes the
radical $\sqrt{I}$, a step which computes a reduced Gr\"obner basis for both the
ideal and its radical (in the same term order), and a step which checks if $I$
and $\sqrt{I}$ have the same reduced Gr\"obner basis. For the step which
computes the radical of $I$, the best known algorithms have (worst case) bounds
{\em doubly exponential} in the number of variables or in the dimension of the
ideal $I$, e.g.,~the algorithm of \cite{laplagne2006algorithm} has complexity
$(rd)^{2^{\mathcal{O}(n^2)}}$, and that of \cite{krick1991algorithm} has
complexity doubly exponential in the dimension of $I$; other algorithms have
similar or worse bounds in the general case. In certain special cases, such as
when $I$ is unmixed \cite[Proposition~4.1]{krick1991algorithm}, $I$ is a
complete intersection \cite{armendariz1995computation}, or $I$ has dimension
zero or one \cite{krick1991membership}, the dedicated algorithms have single
exponential complexity bounds (in the number of variables $n$). The complexity
of computing the Gr\"obner basis for the second step of testing if $I$ is
radical is analogous to radical computation, though in practice the computation
of $\sqrt{I}$ is often much more difficult than computing a reduced Gr\"obner
basis for $I$. Another approach of testing if an ideal is radical is to use the
algorithm \cite{eisenbud1992direct} that computes the primary decomposition, see
also \cite{GiTrZa-prim-decomp-88,ShYo-prim-decomp-96,DeGrPf-prim-decomp-99}.
This approach also has doubly exponential worst case complexity in the number of variables.

\subsection*{Terminology and Notation}
Since our algorithm arises from geometric ideas in intersection theory we will
frequently find it useful to employ more geometric terminology (as we have in
the introduction); we now make this terminology precise. We will, for the most
part, work over an algebraically closed field of chacterisitic zero which we
will denote $k$; usually this will be $\CC$ or the algebraic closure of the
rationals, $\overline{\QQ}$. In particular, given a polynomial ideal $I$ in
$k[x_0,\dots, x_n]$ we will think of it as defining a subscheme $X$ of either
$\pp^n$ if $I$ is homogeneous, or $k^{n+1}$ if not; in both cases we will write
$\mathbb{V}(I)$ for this scheme $X$ defined by $I$. A variety will be a reduced
scheme (we do not assume that a variety is irreducible). When the ideal $I$ has
primary decomposition $I=\mathfrak{q}_1\cap \cdots \cap \mathfrak{q}_r$, then we
will refer to the schemes $\mathbb{V}(\mathfrak{q}_i)$ as {\em primary
  components} of $X$. Similarly we will refer to the scheme
$\mathbb{V}(\mathfrak{q}_i)$ as an {\em isolated primary component}, if
$\mathfrak{p}_i=\sqrt{\mathfrak{q}_i}$ is a minimal prime. The irreducible
varieties $\mathbb{V}(\mathfrak{p}_i)$ for minimal primes $\mathfrak{p}_i$ will
be called {\em irreducible components}, while the irreducible varieties
$\mathbb{V}(\mathfrak{p}_j)$ for embedded primes
$\mathfrak{p}_j=\sqrt{\mathfrak{q}_j}$ will be called {\em embedded components}.
We write $X_{\rm red}$ for the reduced subscheme associated to the scheme
$X$, i.e.,~$X_{\rm red}$ denotes the variety defined by $\sqrt{I}$.

\paragraph{Outline of the paper.}
In the next section, Section \ref{sec:background}, we present the necessary background for Segre class computations and  summarize symbolic methods for sampling at least one point in each irreducible component of a subscheme of afffine space. Section~\ref{sec:collection} contains a collection of various results needed by our algorithm. In particular we present the connection of Segre classes with the integral closure of ideals, we show how to compute the degree of the isolated primary components of a scheme, and how to compute the multiplicities.
Finally, in Section~\ref{sec:main_algorithm} we present our algorithm along with its proof correctness and its complexity analysis.

\section{Background}\label{sec:background}
We briefly review several important concepts and results which will be used
extensively in later sections. In particular we review the notion of Segre
classes and their association to the so called {\em algebraic } or {\em
  Hilbert-Samuel multiplicity}. We will also review how we can sample and
represent symbolically points on a variety using the {\em Rational Univariant
  Representation (RUR)}.

\subsection{Segre Classes and Algebraic Multiplicity}
\label{sec:Segre-intro}

The algorithm that we present in Section \ref{sec:main_algorithm} makes extensive use of ideas from (computational) intersection
theory \cite{Fulton}, and in particular of Segre classes. Below we give
a brief overview of the relevant objects.

In general the Segre class $s(X,Y)$ is defined for pairs of schemes $X\subset Y$
and is an element of the {\em Chow group} of $X$ (see \cite[\S4]{Fulton} for
details). We will restrict our discussion to the case where $X$ and $Y$ are
subschemes of a projective space $\pp^n$ and will instead (via pushforward)
consider the Segre class $s(X,Y)$ as an element of the {\em Chow ring} of the
ambient projective space $\pp^n$, $A^*(\pp^n)$. More explicitly, if $H$ is the
rational equivalence class of a hyperplane in $\pp^n$, then we will represent
elements in the Chow ring $A^*(\pp^n)$ as polynomials in $H$ with integer
coefficients via the isomorphism $A^*(\pp^n)\cong \ZZ[H]/(H^{n+1})$; in this
representation the rational equivalence class of an irreducible variety $V$ of
codimension $c$ is $[V]=\deg(V)\cdot H^c\in A^*(\pp^n)$, where $\deg(V)$ denotes
the {\em degree} of the variety $V$. Hence, in our setting the Segre class
$s(X,Y)$ will be represented as a polynomial of degree at most $n$ in $H$ with
integer coefficients.

Consider subschemes $X$ and $Y$ of $\pp^n$ and let ${\rm Bl}_XY$ denote the blowup of
$Y$ along $X$. This comes equipped with a map $\pi:{\rm Bl}_XY\to Y$ and an
exceptional divisor $E=\pi^{-1}(X)$. In the case where $Y$ is a irreducible variety and $I_X=(f_0,\dots, f_r)$ is the ideal defining the
scheme $X$ we have that ${\rm Bl}_XY $ is isomorphic to the graph $\Gamma$ of
$f_0, \dots, f_r$. More specifically, the graph
is $$ \Gamma=\overline{\{(y,z)\:|\: y\in Y,\;\; z=(f_0(y):\cdots :f_r(y))\}}\subset \pp^n \times \pp^r,
$$ and $\Gamma\cong {\rm Bl}_XY$. In this setting $\pi$ is the projection map $\Gamma\to Y$
from the graph onto $Y$ and the exceptional divisor $E$ is the inverse image of
$X$ under this projection map. The blowup (and hence the graph) captures
information about how $X$ sits inside $Y$ and in particular quantifies how
singular the embedding of $X$ in $Y$ is.

Informally speaking, the Segre class attempts to extract key parts of this
information by considering how $E=\pi^{-1}(X)$ intersects with a similar scheme
which is perturbed to be in general position inside of $\Gamma$. How `general'
of a position we can put a version of $\pi^{-1}(X)$ into inside $\Gamma$ is
significantly determined by how closely $X$ and $Y$ are related. We now give a
formal definition of the Segre class $s(X,Y)$.
\begin{definition}
  Let $X$ and $Y$ are subschemes of $\pp^n$.  We have a blowup diagram
  \[\begin{tikzcd}
    E \ar[d,"\eta"] \ar[r] \arrow[dr, phantom, "\square"] & \mathrm{Bl}_X Y \ar[d,"\pi"] \\
    X \ar[r] & Y \rlap{\ ,}
    \end{tikzcd} \]
    where $E$ is the exceptional divisor. The \emph{Segre class} of $X$ in $Y$ is
  \[
    s(X,Y) = \eta_*((1 - E + E^2 - \dots) \frown [E]) \in A_*(X).
  \]
We will abuse notation and also write $s(X,Y)$ for the pushforward to $A^*(\pp^n)$. \label{def:Segre}
\end{definition}

We now briefly review results of \cite{HH2019segre} which give a explicit and
computable expression for the Segre class. Using the notation above, consider
the rational map $pr_X:Y\to \pp^r$ which is defined by
$pr_X:p\mapsto(f_0(p):\cdots f_r(p))$. We then have the following diagram:
\begin{equation}
\begin{tikzcd}
  & \mathrm{Bl}_X Y \ar[dr,"\rho"] \ar[dl,swap,"\pi"] \arrow[draw=none]{r}[sloped,auto=false]{\subset} & \pp^n \times \pp^r \\
  Y \ar[rr,dashed,"pr_X"] & & \pp^r
\end{tikzcd}\label{eq:SegreDiagBackGrd}
\end{equation}
Computationally, rather than trying to understand the self intersections of $\pi^{-1}(X)$ we will instead seek to understand $pr_X^{-1}(\pp^{r-\dim(Y)-i})-X$ for $i=0,\dots, \dim(Y)$. In particular we will study the {\em projective degrees}, $\pdg_i(\cV, \cW)$, which are the coefficients appearing in the class
$$
G(X,Y)=\sum_{i=0}^{\dim(Y)} [pr_X^{-1}(\pp^{r-\dim(Y)-i})-X] =\sum_{i=0}^{\dim(Y)} \pdg_i(\cV, \cW) H^{n-i}.
$$

We now give explicit expressions for these objects in terms of polynomial ideals in $k[\bm{x}] = k[x_0,\dots, x_n]$, the homogenous coordinate ring of $\pp^n$. As above, $k$ is an algebraically closed field of characteristic zero.

\begin{definition}[Projective Degrees]
  \label{Def:ProjDeg}
  Consider the subschemes  $\cV \subset \cW  \subset \pp^n$.
  Let $\cV$ be defined by the
  homogeneous ideal $I_{\cV}=(f_0,\dots, f_r)$
  and let $d$ be the maximum
  degree of the defining polynomials of $\cV$.
  We can, without
  loss of generality, assume that $\deg(f_i)=d$ for all $0 \leq i \leq r$.
  Define
  the projection of $\cW$ along $\cV$ as the rational map
  \begin{equation}
    \label{eq:piX}
    \begin{array}{lllll}
      \pi_{\cV}: & \cW & \dashrightarrow & \pp^r \\
            & \bm{p}&  \mapsto &  (f_0( \bm{p}):\cdots: f_r(\bm{p})) .
    \end{array}
  \end{equation}
  The projective degrees of $\pi_{\cV}$
  are the sequence of integers $(\pdg_0(\cV, \cW), \dots, \pdg_{\dim(Y)}(\cV, \cW))$, where
  \begin{equation}
    \label{eq:pdeg-1}
    \pdg_i(\cV, \cW)=
    \deg\left( \pi_X^{-1}\big(\pp^{r-(\dim(\cW)-i)}\big) - \cV \right) .
    \end{equation}
Equivalently (via \cite[Proposition~3.3]{HH2019segre}), we can define the projective degrees of
$\pi_X$ as
\begin{equation}
  \label{eqProjDegSchemes}
  \pdg_i(\cV, \cW)=\deg(\cW \cap L^{(i)} \cap \cU - \cV),
\end{equation}
where
$\cU= \VV(P_1, \dots, P_{\dim(\cW)-i})$
with $ P_j = \sum_{\nu=0}^{r} \lambda_{\nu} f_{\nu}$
for generic $\lambda_{\nu}\in k$
and  $L^{(i)}\subset \pp^n$ is a
generic linear space of codimension $i$.
\end{definition}

Using the latter characterization of projective degrees, we can express them
with respect to the ideals corresponding to the subschemes
$\cV \subset \cW \subset \pp^n$, that is $I_{\cV} = \langle f_0,\dots, f_r\rangle$ and $I_{\cW}$.
Let $t$ be a new variable; then, for $0 \leq i \leq r$, we define the
family of ideals
\begin{equation}
  \label{eq:ProjDegIdealEqs}
  \mathcal{I}_i= I_{\cW} +
  \Big\langle\sum_{j=0}^r \lambda_{1,j}f_j,\dots, \sum_{j=0}^r \lambda_{\dim(Y)-i,j}f_j,
  \ell_1(\bm{x}),\dots,\ell_i(\bm{x}),\ell_0(\bm{x})-1 ,
  1 -t\sum_{j=0}^r \lambda_{0,j}f_j \Big\rangle
  \subset k[\bm{x}][t],
\end{equation}
where $\ell_\nu(\bm{x})=\sum_{j=0}^n\theta_{\nu,j}x_j$ for generic
$\theta_{l,j}\in k$, and generic $\lambda_{l,j}\in k$.  We also write
$\mathcal{I}_i(\cV, \cW)$ to denote the dependency on the subschemes
$\cV$ and $\cW$.
By \cite[Theorem~3.5]{HH2019segre} the projective degree of dimension $i$ can be computed as
\begin{equation}
  \pdg_i(\cV, \cW) = \dim_k \left( k[\bm{x}, t] / \mathcal{I}_i\right).\label{eq:ProjDegIdeal}
\end{equation}
In \cite[\S3]{HH2019segre} an explicit formula for the Segre class $s(X,Y)$ is given which depends only on the numbers $\pdg_i(\cV, \cW)$ obtained via the computation in \eqref{eq:ProjDegIdeal} and the degree $d$ of the generators of $I_X$ (though the final Segre class does not depend on $d$). The projective degrees can also be used to compute the {\em algebraic multiplicity} of a variety inside a scheme. 

Let $X\subset \pp^n$ be an irreducible (and reduced) subvariety and
let $Y\subset \pp^n$ be a pure dimensional subscheme with corresponding ideals $I_X=\langle f_0, \dots, f_r \rangle$ and $I_Y$ in $k[\bm{x}]=k[x_0,\dots, x_n]$. The
\textit{algebraic or Hilbert-Samuel multiplicity} of $X$ on $Y$, denoted $e_XY$ is the
integer coefficient of $[X]$ in the Segre class $s(X,Y)$. This multiplicity is more classically defined via the Hilbert-Samuel polynomial of the local ring $(k[x_0,\dots, x_n]/I_Y)_{I_X}$, where the subscript denotes localization at the prime ideal $I_X$, see, e.g., \cite[Example~4.3.1,Example~4.3.4]{Fulton} or \cite[Chapter~12]{eisenbud2013commutative}.  In practice we will compute this multiplicity as follows. Let $d$ be
the maximum degree among a set of generators of $I_X$. By
\cite[Theorem~5.2]{HH2019segre} we have that
\begin{equation}
  e_XY=\frac{\deg(Y) d^{\dim(Y)-\dim(X)} - \pdg_{\dim(X)}(X,Y)}{\deg(X)},
  \label{eq:Multiplicity}
\end{equation}
where $\pdg_{\dim(X)}(X,Y)$ is the dimension $X$ projective degree of
$X$ in $Y$, see \eqref{eq:ProjDegIdeal},
\eqref{eqProjDegSchemes}, or Definition \ref{Def:ProjDeg}.

It is a classical result of Samuel \cite[II §6.2b]{samuel1955methodes} (see also \cite[Ex.~12.4.5(b)]{Fulton}) that $e_XY=1$ if and only a generic point in $X$ is not contained in the singularity subscheme of $Y$. We state this as a proposition below. 
\begin{proposition}[Samuel \cite{samuel1955methodes}]
  Let $X\subset \pp^n$ be an irreducible subvariety and
let $Y\subset \pp^n$ be a pure dimensional subscheme. Then $e_XY=1$ if and only if a generic point in $Y$ is reduced and $X$ is not contained in the singular locus of $Y_{\rm red}$. \label{prop:RadicalTestIrreducibleCase}
\end{proposition}
This fact along with a formula derived in \S\ref{sec:mult-subvariety} will be adapted to furnish a test which tells us when an isolated primary component of a scheme is generically reduced. 

In the final section we will additionally need to work with varieties $X$ which are not irreducible, in this case it will be most convenient to write our criterion in terms of the $\dim(X)$ part of the Segre class. Let $\cW$ be a pure dimensional subscheme of $\pp^n$ and let $\cV$ be a closed
  subscheme of $\cW$. Let $d$ be the maximum degree of the equations
  defining $\cV$. With the notations above the result of \cite[Corollary~ 3.14]{HH2019segre} gives the following: 
\begin{equation}
    \{s(\cV, \cW)\}_{\dim(X)}=
    d^{\dim(\cW)-\dim(\cV)} \, \deg(\cW) - \pdg_{\dim(\cV)}(\cV, \cW).\label{eq:SegreDimX}
\end{equation} 

\subsection{Rational Univariate Representation and Computing One Point in Each Irreducible Component}
\label{sec:rur}

An important part of the  main algorithm presented in Section \ref{sec:main_algorithm} is the ability to sample points from each irreducible component of a scheme $X\subset \CC^n$. Approximate point samples could be furnished using methods such as homotopy continuation from numeric algebraic geometry \cite{sommese2005numerical}, we however opt for symbolic methods, e.g.,~\cite{Rojas-sparse-rur}, in our presentation as these methods have well understood worst case complexities. We note that our algorithm certainly could be implemented using numerical methods instead, see also Remark \ref{remark:NumericVersion}.
In the symbolic setting, it is important for the sampling algorithms we employ to have an efficient and exact representation of the sampled points. 
For this  we exploit the  {\em rational univariate representation (RUR)}
\cite{Rou:rur:99}, see also \cite{AlBeRoWo-idem-96,MaScTs-srur-17}. 

Briefly,
given a zero dimensional ideal in $\QQ[x_1, \dots, x_n]$,
RUR represents the coordinates of the associated points as a univariate rational function
evaluated at the roots of univariate polynomial. 
It is of the form
$R(t), x_1 = \tfrac{R_1(t)}{R_0(t)}, \dots, x_n = \tfrac{R_n(t)}{R_0(t)}$
where $R, R_0, R_1, \dots, R_n \in \QQ[t]$ and $t$ is a new variable.
The following theorem provides a brief presentation of RUR and
summarizes some of its important properties.

We note in the Theorem statement below we consider the (reduced) zero set associated to a polynomial system (or polynomial ideal). These methods do not require that the input polynomials define a radical ideal but will only furnish information about the zero set of the input polynomials, i.e.,~about the associated reduced subscheme. 
\begin{theorem}
  \label{thm:RUR}
  Consider the solution set 
  $W=\{ x\in \CC^n \;|\; f_1(x) = \cdots = f_n(x) = 0 \}$, where the polynomials
  $f_i \in \ZZ[x_1, \dots, x_n]$ are dense of degree $d$ and maximum
  coefficient bitsize $\tau$. 
  Then, there is an algorithm that computes univariate
  polynomials of $R, R_1, \dots, R_n \in \ZZ[t]$ such that:
  \begin{enumerate}
  \item The degrees of $R, R_1,\ldots, R_n$ are all bounded by
    $d^n$ and their bitsize by $\OO(d^n + n d^{n-1} \tau)$.
  \item For any root $\xi$ of $R(t) = 0$, the tuple
    $r(\xi) := \left(\frac{R_1(\xi)}{R'(\xi)}, \dots, \frac{R_n(\xi)}{R'(\xi)}\right) \in (\CC^{*})^n$
    is a point in $W$; $R'$ is the derivative of $R$ with respect to $t$.
    We call this representation the Rational Univariate Representation (RUR) \cite{Rou:rur:99}.
    \item Every $r(\xi)$ corresponds to an isolated point of $W$
    or to a point on an irreducible component of $W$ (of dimension greater than 0).
    In addition, the set of points
    $\{ r(\xi) \in \CC^n \,|\, R(\xi)=0 \}$ contains all the
    0-dimensional irreducible components of $W$ in $\CC^n$.
    \item  We compute the RUR by performing $d^{\OO(n)}$
  arithmetic operations.
\end{enumerate}
\end{theorem}
\begin{proof}
  For the bounds on the bitsize of the coefficients and a
  generalization to the sparse case and how we compute the RUR using resultants  we refer the reader to
  \cite[Theorem~4.3]{MaScTs-srur-17}.
  For item 3~we refer the reader to \cite{Rojas-sparse-rur}.

  To compute the RUR in the case where the system is not 0-dimensional we use Canny's generalized characteristic polynomial \cite{Canny-GCP-90}, or
  the toric variant by Rojas \cite{Rojas-sparse-rur,rojas-arxiv-1997}, that relya
  on Macaulay and sparse resultant matrices, respectively, see \cite{MaScTs-srur-17}. The complexity bound follows
  from \cite{Canny-PSPACE}.
\end{proof}

We remark that the above result is presented in the case of a square system, however if the system is not square, then one may simply take a general linear combination of the generators to get a square system and verify which sampled points satisfy the original system. To get samples from all irreducible components of higher dimensions we may then add a generic linear polynomial to the defining ideal (i.e.,~intersect the scheme with a generic hyperplane) and repeat this procedure. Such considerations are discussed in more detail in the references.

The algorithms (and the corresponding mathematical results) that support
Theorem~\ref{thm:RUR} allow us to  sample at
least one (generic) point from each irreducible component of an arbitrary subscheme of $\CC^n$. The symbolic algorithms that
support these computations rely on resultant or Gr\"obner basis computations and
careful analysis of the bitsize of the involved polynomial computations. We
refer the interested reader to \cite{elkadi1999new} for an approach based on
B\'ezoutian matrices, to \cite{MaScTs-srur-17, Rojas-sparse-rur} for an approach based on
resultant matrices and to
\cite{keyser2005exact} for an implementation. This leads us to the following corollary.
\begin{corollary}
  Let $I$ be a polynomial ideal generated by polynomials of degree at most $d$
  in $\QQ[x_1,\dots, x_n]$ defining a scheme $X\subset\CC^n$. Then, we can obtain
  a collection of RUR of points containing at least one generic point in each
  irreducible component of $X$ in at most $d^{\mathcal{O}(n)}$ arithmetic
  operations. \label{cor:RUR}
\end{corollary}

\section{A Collection of Results Which Enable Our Algorithm}
\label{sec:collection}

In this section we gather together several results which will be employed by our main algorithm which is presented in Section \ref{sec:main_algorithm}. We first consider, in \S\ref{subsec:Segre_IntegralClosure}, some results regarding the relationships between the radical of an ideal, its integral closure, and a Segre class associated to the corresponding scheme. Following this, in \S\ref{section:degree_Isolated_primary_comp_dim_mu}, we show how the degree of isolated primary components of any dimension can be computed in worst case singly exponential time with respect to the number of variables. Then, in \S\ref{subsec:RadicalPrimeIsolateComps}, we give a result which lets us compute the radical of an ideal in one saturation when all isolated primary components of the associated scheme are reduced. Finally, in \S\ref{sec:mult-subvariety} and \S\ref{subsec:affineMult}, we look at multiplicity and (partial) Segre class computations relative to isolated primary components of any dimension and how these computations can be done in a generic affine patch of projective space.   
\subsection{Segre Classes and the Integral Closure of Ideals}
\label{subsec:Segre_IntegralClosure}
We will now give a test which will let us determine if an ideal $I$ is equal to
its radical (upto a factor of the irrelevant ideal) via a Segre class
computation. En-route we shall be led to consider the {\em integral closure}
$\overline{I}$ of the ideal $I$. We shall in particular be interested in the
case where $\overline{I}=\sqrt{I}$.

We begin by reviewing several key results from the literature. First, we present
a result of Gaffney and Gassler \cite{gaffney1999segre}, that we slightly modify
to fit in our notational framework. This result makes clear the connection
between Segre classes and integral closures of ideals. One can also deduce this
result from other more recent works, such as
\cite{aluffi2017many,aluffi2017segre} that use a more similar terminology to
ours.

\begin{proposition}[Corollary 4.9 of \cite{gaffney1999segre}]
  Let $X$ and $Y$ be subschemes of $\pp^n$ defined by ideals $I_X$ and $I_Y$ in
  $k[x_0,\dots, x_n]$ with $X_{\rm red}=Y_{\rm red}$. Then, we have that
  $s(X,\pp^n)=s(Y,\pp^n)$ if and only
  if $$\overline{I_X}:\langle x_0,\dots, x_n \rangle^\infty=\overline{I_Y}:\langle x_0,\dots, x_n \rangle^\infty.$$\label{propn:SegreIntClosure}
\end{proposition}
To connect the result above to the problem of checking if an ideal is equal to
its radical we will need to in turn explore connections between the radical of
a polynomial ideal and its integral closure. Our main sources here are
\cite{vasconcelos2006integral} and \cite{huneke2006integral}. It is shown in
\cite[Remark~1.6]{vasconcelos2006integral} that if $J$, $L$, and $W$ are ideals
in an integral domain $R$ with $\overline{J}=\overline{L}$, then we have that
$(JW):L\subseteq \overline{W}$. In our setting this gives the following
corollary.

\begin{corollary}
  Let $R=k[x_0,\dots, x_n]$ and let $I$ be any proper ideal such that $\sqrt{I}=\overline{I}$. Then we have that:\begin{enumerate}
      \item $(\sqrt{I})^2:I\subset \sqrt{I}$,
      \item $(\sqrt{I})^2 \subseteq I$.
  \end{enumerate}\label{cor:radSquaredInI}
\end{corollary}
\begin{proof}
Part 1 follows immediately from \cite[Remark~1.6]{vasconcelos2006integral} and the fact that $\sqrt{I}=\overline{I}$ by taking $J=W=\sqrt{I}$ and $L=I$ (note $\overline{L}=\overline{J}=\sqrt{I}$ by assumption).

Since $I$ is a proper ideal of $R$ it follows that $\sqrt{I}$ is also a proper ideal, and hence part 2 follows by properties of the colon ideal.
\end{proof}

Below we state a characterization of the integral closure of an ideal in a Noetherian ring.
\begin{proposition}[Corollary~6.8.11 of \cite{huneke2006integral}]
Let $R$ be a Noetherian ring, $I$ and ideal in $R$ and $r$ any element of $R$. Then $r\in \overline{I}$ if and only if there exists an integer $\nu$ such that for all integers $m>\nu$ we have $r^m\in I^{m-\nu}$.\label{prop:IntClosueCharcter}
\end{proposition}

We now show that an ideal whose radical and integral closure agree must be radical, this will allow us to directly employ Proposition \ref{propn:SegreIntClosure} to test if an ideal is equal to its radical. 

\begin{proposition}
Let $R=k[x_0,\dots, x_n]$ and let $I$ be any ideal such that $\overline{I}=\sqrt{I}$. Then $I=\sqrt{I}$, that is $I$ is radical.
\end{proposition}\begin{proof}
Suppose that $I$ is not radical but $\overline{I}=\sqrt{I}$; we will show this yields a contradiction. If $I\neq \sqrt{I}$ then $\sqrt{I}-I$ is non-empty, hence we may choose some $r\neq 0$ in $\sqrt{I}-I$. Then we have that $r\in \sqrt{I}=\overline{I}$, but $r\not\in I$. Further by Corollary \ref{cor:radSquaredInI} we have that $\sqrt{I}^2\subseteq I$, hence $r^2\in I$. Since $R$ is a UFD we have $r=r_1^{a_1}\cdots r_c^{a_c}$ is the unique expression for $r$, further we may choose $r\in \sqrt{I}-I$ minimal in that sense that $r\in \sqrt{I}$, but $(r_1^{a_1}\cdots r_c^{a_c})/r_i\not \in \sqrt{I}$ for any $i$; we may choose $r$ in this way since if one of these expressions were in $\sqrt{I}$ we could simply take this factor as $r$ instead.

{\bf Claim:} {\em For any $\nu \geq 0$ we have $r^m\not\in I^{m-\nu}$ for $m>2\nu$.} Note that, if this claim holds then by Proposition \ref{prop:IntClosueCharcter} we have that $r\not\in \overline{I}=\sqrt{I}$, but this contradicts our assumption that $\sqrt{I}-I$ is non-empty.

We now prove the above claim: {\em for any $\nu \geq 0$ we have $r^m\not\in I^{m-\nu}$ for $m>2\nu$.} Since $r\not \in I$ but $r^2=r_1^{2a_1}\cdots r_c^{2a_c}\in I$, then $r^4=r_1^{4a_1}\cdots r_c^{4a_c}\in I^2$ is the smallest power of $r$ in $I^2$, since $R$ is a UFD. In particular note that if $r^3=r_1^{3a_1}\cdots r_c^{3a_c}$ were in $I^2$ then this would imply $r\in I$ since we know that $(r_1^{a_1}\cdots r_c^{a_c})/r_i\not \in \sqrt{I}$ for all $r_i$. Similarly $r^6$ is the smallest power of $r$ in $I^3$, and by induction, $r^\ell \not \in I^\mu$ for $\ell<2\mu$. Now take $\ell=m$, $\mu=m-\nu$, then we have that $r^m\not\in I^{m-\nu}$ for any $m$ and $\nu$ such that $m<2(m-\nu)=2m-2\nu$, or equivalently, such that $2\nu <m$. In other words, for any $\nu \geq 0$ we have $r^m\not\in I^{m-\nu}$ for $m>2\nu$; which is what we wished to show. Hence the claim is proved and the result follows as discussed above.
\end{proof}

This result, along with Proposition \ref{propn:SegreIntClosure} combine to yeild a method to test if an ideal is equal to its radical {\em without} using a Gr\"obner basis. We summarize this in the corollary below.

\begin{corollary}
  Let $I$ be a homogeneous polynomial ideal in $k[x_0, \dots, x_n]$, let $X\subset \pp^n$ be the scheme defined by $I$ and let $X_{\rm red}$ be the reduced scheme defined by $\sqrt{I}$. Then $I:\langle x_0,\dots, x_n \rangle^\infty$ is radical if and only if $s(X,\pp^n)=s(X_{\rm red}, \pp^n)$. \label{cor:test_ideal_equal_radical}
\end{corollary}
We note that while the result of Corollary \ref{cor:test_ideal_equal_radical} removes the need to use a Gr\"obner basis calculation (or some other doubly exponential method) to test if $\sqrt{I}$ is equal to $I$ (up to saturation by the irrelevant ideal), it does not immediately remove the requirement to compute $\sqrt{I}$. In many cases $s(X_{\rm red}, \pp^n)$ can in fact be computed with objectively less information than a generating set of $\sqrt{I}$, 
in general however it is not clear how to compute $s(X_{\rm red}, \pp^n)$ without a generating set for $\sqrt{I}$. The result presented in \S\ref{subsec:RadicalPrimeIsolateComps} reduces this problem (in the situations where we need to solve it) to that of computing $s(\VV(I:J^\infty), \pp^n)$, for $J$ generated by the ${\rm codim}(X)\times {\rm codim}(X)$ minors of the Jacobian of $I$, using only a generating set of $I$ and $J$; as yet we are also unable to furnish this computation however. This remains the main barrier to the algorithm presented in Section \ref{sec:main_algorithm} having a singly exponential worst case bound (in the number of variables) in all cases.

\subsection{The Degree of Isolated Primary Components of a Scheme in Singly Exponential Time}\label{section:degree_Isolated_primary_comp_dim_mu}
In this subsection we describe how we can use methods to compute a geometric equidimensional decomposition such as \cite{jeronimo2002effective}, along with zero dimensional Gr\"obner basis computation, to obtain the sum of the degrees of all isolated primary components of a scheme of a given dimension. Given an ideal $I=\langle f_1, \dots, f_r \rangle$ in the polynomial ring $\ZZ[x_1,\dots, x_n]$ where $\deg(f_i)+1\leq d$ for all $i$ which defines a scheme $X=\mathbb{V}(I)\subseteq \CC^n$, we will show in Proposition \ref{prop:Deg_Low_dim_Complexity}, that this degree computation has worst case complexity bounded by $d^{\mathcal{O}(n^3)}$. We note that this computation takes as input only the generators of the ideal $I$.

The algorithm of  \cite{jeronimo2002effective} takes as input the ideal $I$ and outputs (a straight-line program) defining (distinct) polynomial ideals $K_{i_1}, \dots, K_{i_m}$ where these ideals give an equidimensional decomposition of the variety $X_{\rm red}=\VV(\sqrt{I})$, i.e.,~$$
X_{\rm red}=(\VV(K_{i_1}))_{\rm red}\cup \cdots \cup (\VV(K_{i_m}))_{\rm red}, 
$$where $(\VV(K_{i_j}))_{\rm red}$ is the union of all irreducible components of $X_{\rm red}$ of dimension $i_j$.

Without loss of generality we may assume that $i_1>\cdots >i_m$. Let $L_j$ be a general linear form in  $\CC[x_1,\dots, x_n]$. Suppose that $K_{i_j}=\langle g_1, \dots, g_r\rangle$ and define the polynomial $P(K_{i_j},T)$ in $\CC[x_1, \dots, x_n,T]$ as \begin{equation}
P(K_{i_j},T):=1-T\cdot \sum_{\nu=1}^r \lambda_\nu g_\nu.\label{eq:P_equi_decomp}
\end{equation} Let $i_{\nu}\geq \mu \geq 0$ where $\nu$ is minimal and $i_1\geq \mu\geq 0$ holds. Define the {\em (isolated) degree in dimension $\mu$} of $I$ as \begin{equation}
\deg_\mu(I):=\dim_{\CC} \CC[x_1, \dots, x_n,T_1, \dots, T_{\nu}]/\mathfrak{J}(\mu) ,
\label{eq:deg_mu_def}
\end{equation}
where
$$
\mathfrak{J}(\mu):=I+\langle L_1, \dots, L_{\mu} \rangle+\langle P(K_{i_1},T_1), \dots,P(K_{i_\nu},T_\nu)  \rangle.
$$ If $X$ is the subscheme of $\CC^n$ defined by $I$ we will define $$
\deg_\mu(X):=\deg_\mu(I).
$$

\begin{proposition}
 Let $I$ be a polynomial ideal in $\QQ[x_1,\dots, x_n]$. Suppose that the ideal $\mathfrak{I}_\mu$ is the intersection of all $\mu$-dimensional isolated primary components of $I$. Then, with $ \deg_\mu(I)$ as in \eqref{eq:deg_mu_def}, we have  $$
 \deg_\mu(I)=\deg(\mathfrak{I}_\mu).
 $$
\end{proposition}\begin{proof} Note that this result is essentially an application of the standard Rabinowitsch trick. Let $X$ be the scheme defined by $I$ in $\CC^{n+\nu}$. 
  By construction all primary components of $X$ of dimension greater than $\mu$ have an empty intersection with the scheme defined by  $\langle P(K_{i_1},T_1), \dots,P(K_{i_\nu},T_\nu)  \rangle$, note this also applies to any embedded components of $X$ of any dimension which are contained in isolated primary components of $X$ of dimension greater than $\mu$. Note for points $(x_1,\dots, x_n)$ not in $(\mathbb{V}(K_{i_l}))_{\rm red}$ the polynomial $P(K_{i_l}, T)$ has a single solution of multiplicity one, hence intersection with these hypersurfaces can be thought of a intersection with hyperplanes for isolated primary components of $X$ of dimension less than or equal to $\mu$. Since the linear forms $L_j$ are general and there are $\mu$ of them the conclusion follows since the only points which remain in the zero dimensional scheme defined by $\mathfrak{J}(\mu)$ came from those in the scheme defined by $\mathfrak{I}_\mu$.
\end{proof}

Finally we note that the expression in \eqref{eq:deg_mu_def} can be computed in singly exponential time in the number of variables. More specifically we have the following result.

\begin{proposition}\label{prop:Deg_Low_dim_Complexity}
Consider a polynomial ideal $I=\langle f_1, \dots, f_r \rangle$ in the polynomial ring $\ZZ[x_1,\dots, x_n]$ where $\deg(f_i)+1\leq d$. For any $\mu\leq \dim(I)$, the worst case complexity of computing the expression $\deg_\mu(I)$ is bounded in $d^{\mathcal{O}(n^3)}$.
\end{proposition}\begin{proof}
  From results in \cite{jeronimo2002effective} we know the degree of the polynomials $g_\nu$ appearing in \eqref{eq:P_equi_decomp} is bounded by $d^{n^2}$.
  The output of \cite{jeronimo2002effective} consists of  polynomials in straight-line programs representation. However, using evaluation/interpolation we can convert them to a dense representation, using the bound on their degree.

  Since the ideal $\mathfrak{J}(\mu)$ has dimension zero by construction, using standard bounds on the computation of Gr\"obner basis in dimension zero \cite{hashemi2005complexity,lakshman1991single} (or for other methods to compute the degree of a zero dimensional ideal) gives $(d^{n^2})^{2n-1}$, giving $d^{\mathcal{O}(n^3)}$.
\end{proof}



\subsection{Finding the Radical of an Ideal with Prime Isolated Primary Components}\label{subsec:RadicalPrimeIsolateComps}
In this subsection we
consider a scheme $X$ such that all isolated primary components are reduced and
prove a compact formula in terms of a certain saturation for the ideal of the
reduced scheme $X_{\rm red}$. The following proposition is a summary of the
result we will use in the algorithm in Section \ref{sec:main_algorithm}, this result is
essentially just by definition, and no doubt well known, but we include a short
proof for completeness.

\begin{proposition}
Let $I$ be a homogeneous ideal in $k[x_0,\dots, x_n]$ defining a scheme $X=\mathbb{V}(I)\subset \pp^n$ with $c={\rm codim}(X)$. Let $J=I+K$ where $K$ is the ideal generated by all $(c\times c)$-minors of the Jacobian matrix of $I$. Suppose that all isolated primary components of $X$ are reduced, then $X_{\rm red}=\mathbb{V}(I:J^\infty)$; i.e.,~$\sqrt{I}=I:J^\infty$ up to saturation by the irrelevant ideal. In other words $X_{\rm red}=\overline{X-\mathbb{V}(J)}$.\label{thm:radical_isolated_comps_reduced}
\end{proposition}
\begin{proof}
  Relabeling a primary decomposition of $X$ we may write
  $$
X=(W^{\rm iso}_1\cup W^{\rm emb}_1)\cup \cdots \cup(W^{\rm iso}_r\cup W^{\rm emb}_r),
$$
where $W^{\rm iso}_i$ is an isolated primary component and $W^{\rm emb}_i$ is the union of all primary components which correspond to embedded components contained in the irreducible component corresponding to $W^{\rm iso}_i$. Further since, by assumption, all isolated primary components are reduced we must have that $$
X_{\rm red}=W^{\rm iso}_1\cup \cdots \cup W^{\rm iso}_r.
$$Now consider $\overline{X-\mathbb{V}(J)}$, we have \begin{align*}
    \overline{X-\mathbb{V}(J)}=&\overline{\left((W^{\rm iso}_1\cup W^{\rm emb}_1)\cup \cdots \cup(W^{\rm iso}_r\cup W^{\rm emb}_r)\right) -\mathbb{V}(J)}\\
    =&\overline{(W^{\rm iso}_1\cup W^{\rm emb}_1) -\mathbb{V}(J)}\cup \cdots \cup\overline{(W^{\rm iso}_r\cup W^{\rm emb}_r)-\mathbb{V}(J)}\\
    =&W^{\rm iso}_1\cup \cdots \cup W^{\rm iso}_r.
\end{align*} The last equality follows since $W^{\rm iso}_i$ is irreducible and reduced (being reduced means $W^{\rm iso}_i\cap  \mathbb{V}(J)\subsetneq W^{\rm iso}_i$) and since we must have that $W^{\rm emb}_i\subset \mathbb{V}(J)$ for all $i$. Hence,  $X_{\rm red}=\overline{X-\mathbb{V}(J)}$.
\end{proof}

\subsection{Multiplicity of a Subvariety in Isolated Primary Components}
\label{sec:mult-subvariety}

Using the dimension sensitive degree function from
\S\ref{section:degree_Isolated_primary_comp_dim_mu} and the multiplicity formula
\eqref{eq:Multiplicity} we  obtain a method to compute the multiplicity of an
irreducible variety contained in the union of all $\nu$-dimensional isolated primary components of a
scheme $Y\subset \pp^n$ with multiple components of different dimensions.

Let $Y\subset \pp^n$ be an arbitrary subscheme. Let $W$ be the union of all isolated primary components of $Y$ of dimension  $\nu$ and let $X=\mathbb{V}(f_0, \dots, f_r) \subset \pp^n$ be an irreducible subvariety of $W$ with the generators chosen such that $\deg(f_i)= d$.   
Define the ideal $$
{\mathcal{I}}_{\nu}(X)= {I}_{Y} +  \Big\langle\sum_{j=0}^r \lambda_{1,j}f_j,\dots, \sum_{j=0}^r \lambda_{\nu-\dim(X),j}f_j,
\ell_0(\bm{x})-1 ,
  1 -t\sum_{j=0}^r \lambda_{0,j}f_j \Big\rangle
  \subset k[\bm{x}][t].
$$Since $Y$ has isolated primary components of dimension $\nu$ the scheme $\mathbb{V}(\mathcal{I}_{0,\nu})$ has isolated primary components of dimension equal to $\dim(X)$. Then, since $X\subset W$ we have  \begin{equation}
e_XW=\frac{\deg_{\nu}(Y)d^{\nu-\dim(X)}-\deg_{\dim(X)}({\mathcal{I}}_{\nu}(X))}{\deg(X)}.   \label{eq:MultInComp} 
\end{equation}
We note that this formula follows immediately from \eqref{eq:Multiplicity} since
the functions $\deg_\nu$ and $\deg_{\dim(X)}$, respectively, will capture the
degrees of the $\nu$ and $\dim(X)$ dimensional isolated primary components only, respectively.
Hence, since $X$ is contained in the pure $\nu$-dimensional scheme $W$, the primary components of other dimensions
play no role and we immediately obtain $e_XW$. For identical reasons we obtain
an analogous version of \eqref{eq:SegreDimX} which we state below:

\begin{equation}
    \{s(\cV, W)\}_{\dim(X)}=
    d^{\nu-\dim(\cV)} \, \deg_\nu(\cW) - \deg_{\dim(X)}({\mathcal{I}}_{\nu}(X)).\label{eq:SegreDimX_For_component}
\end{equation} 

We note that these computations take as input only the equations defining the schemes $X$ and $Y$ and that their complexity is determined by that of the degree function of \S\ref{section:degree_Isolated_primary_comp_dim_mu}, see Proposition \ref{prop:Deg_Low_dim_Complexity}.
We can also extend this to the case where $X$ is an arbitrary variety not contained in $W$, but contained in $Y$, using \cite[Example~4.3.4]{Fulton}, or the similar result for Segre classes, \cite[Lemma~4.2]{Fulton}, however this extension will not be needed for our purposes here. It is worth noting, however, that in particular the result of \cite[Lemma~4.2]{Fulton} implies that if $X\cap W$ is empty then $s(X,W)=0$ and that if $\dim(X\cap W)<\dim(X)$ then $\{s(\cV, W)\}_{\dim(X)}=0$, and further it follows from results in \cite{HH2019segre} that all these properties hold for the formula given above, e.g.~\eqref{eq:SegreDimX} and \eqref{eq:SegreDimX_For_component}. 


\subsection{Affine Version of Multiplicity}\label{subsec:affineMult}
From the definition of the Hilbert-Samuel multiplicity in terms of the Hilbert-Samuel polynomial one can easily deduce  that the multiplicity of a projective variety in a subscheme of projective space agrees with the multiplicity of the resulting pair of affine varieties on a generic affine patch. Below we will see that our method for computing these multiplicities works similarly. This will be useful when we need to employ sampling techniques which are meant for affine varieties in our main algorithm.  


Let $X\subset W\subset \pp^n$ with $X$ a variety and $W$ the union of all dimension $\nu$ isolated primary components of $Y$. Suppose that $I_X=\langle f_0, \dots f_r\rangle$ and without loss
of generality assume $d:=\deg(f_i)$ for all $i$. Take
$\ell_0(x)\in \CC[x_0,\dots,x_n]$ to be a general linear form and set
$\hat{I}_X=I_X+\langle\ell_0-1\rangle$ and $\hat{I}_Y=I_Y+\langle\ell_0-1\rangle$. Finally, let
$\hat{Y}\subset \CC^{n+1}$ be the scheme defined by the ideal $I_Y+\langle\ell_0-1\rangle $,
let $\hat{W}$ be the pure $\nu$-dimensional subscheme of $\hat{Y}$ arising from $W$ and let
$\hat{X}=V(\hat{I}_X) \subset \CC^{n+1}$. Note that $\hat{X}$ is a subvariety of
$\hat{W}$ and that the degrees and dimensions are unchanged,
i.e.,~$\deg(X)=\deg(\hat{X})$, $\deg_\mu(Y)=\deg_\mu(\hat{Y})$ for all $\mu$
such that $Y$ has an isolated primary component of dimension $\mu$,
$\dim(X)=\dim(\hat{X})$, and $\dim(Y)=\dim(\hat{Y})$. We can think of this
procedure of moving to the ideal $\hat{I}$ from $I$ as a combination of
dehomogenizing with respect to a general point at infinity and linearly
embedding $\CC^n$ into $\CC^{n+1}$. Now
let $\phi_i=f_i$ for $i=1,\dots, r$ and $\phi_{r+1}=\ell_0-1$, set $$
\hat{\mathcal{I}}_\nu(X)= \hat{I}_{Y} +  \Big\langle\sum_{j=0}^r \lambda_{1,j}\phi_j,\dots, \sum_{j=0}^r \lambda_{\nu-\dim(X),j}\phi_j,\ell_0(\bm{x})-1 ,
  1 -t\sum_{j=0}^r \lambda_{0,j}\phi_j \Big\rangle
  \subset k[\bm{x}][t],
$$ and as in the previous subsection $$
{\mathcal{I}}_\nu(X)= {I}_{Y} +  \Big\langle\sum_{j=0}^r \lambda_{1,j}f_j,\dots, \sum_{j=0}^r \lambda_{\nu-\dim(X),j}f_j,\ell_0(\bm{x})-1 ,
1 -t\sum_{j=0}^r \lambda_{0,j}f_j \Big\rangle
  \subset k[\bm{x}][t].
$$Observe that $\dim(\mathcal{I}_i)=\dim(\hat{\mathcal{I}}_i)$ and that $\deg_\nu(\mathcal{I}_i)=\deg_\nu(\hat{\mathcal{I}}_i)$. 
Hence if we define
\begin{equation}
e_{\hat{X}}\hat{W}:=\frac{\deg_\nu(\hat{Y}) d^{\nu-\dim(\hat{X})} - \deg_{\dim(X)}(\hat{I}_\nu(X))}{\deg(\hat{X})}   ,\label{eq:Multiplcity_Affine_Version}
\end{equation} 
then using \eqref{eq:MultInComp} we immediately obtain that 
$e_{\hat{X}}\hat{W}=e_{{X}}{W}.
$ This equivalence will be used extensively in the sequel to allow us to freely switch between projective and affine formulations. This also immediately gives the following formula for the dimension $\dim(X)$ part of the Segre class of $X$ in $W$:\begin{equation}
\{s(X,W)\}_{\dim(X)}=\deg_\nu(\hat{Y}) d^{\nu-\dim(\hat{X})} - \deg_{\dim(X)}(\hat{I}_\nu(X)),\label{eq:Segre_DimX_AffineForm}    
\end{equation}
where we no longer need to assume that $X$ is irreducible or reduced, we only require $X$ is a subscheme of $W$. Note, in Algorithm \ref{alg:is_radical} we will abuse notation and write $\{s(\hat{X},\hat{W})\}_{\dim(X)}$ to mean the expression $\{s(X,W)\}_{\dim(X)}$ computed via the affine formula \eqref{eq:Segre_DimX_AffineForm} above. 


\section{Testing if an Ideal is Radical}\label{sec:main_algorithm}
In this section we present our main algorithm, Algorithm \ref{alg:is_radical},
for testing if a scheme $Y$ in $\pp^n$ is reduced or, equivalently, if the
defining ideal is radical up to saturation by the irrelevant ideal. First, in
\S\ref{subsec:IdealAlgorithm}, we give the idealized version (i.e.,~we assume
that we can work over $\CC$) and that for a point sampled from an irreducible
component we obtain the associated maximal ideal. Next, in
\S\ref{sec:mult-and-RUR}, we explain how we would adapt Algorithm
\ref{alg:is_radical} to perform computations with (groups of) sampled points
represented in the Rational Univariate Representation (RUR), as it is required
for a realistic symbolic implementation.

\subsection{Algorithm}\label{subsec:IdealAlgorithm}
Here we present our main algorithm, Algorithm \ref{alg:is_radical}, to test if a homogeneous ideal is radical, up to
saturation via the irrelevant ideal. We begin with a proof of correctness. 

\begin{theorem}
  \label{thm:AlgCorrect}
  Given a homogeneous ideal $I$ in $k[x_0,\dots, x_n]$ Algorithm
  \ref{alg:is_radical} correctly tests if
  $I:\langle x_0,\dots, x_n \rangle^\infty$ is radical and termiates in finite time.
\end{theorem}\begin{proof}
  The fact that Line 7 detects all isolated primary components which are non-reduced
  follows from Proposition \ref{prop:RadicalTestIrreducibleCase}, combined with
  \eqref{eq:Multiplcity_Affine_Version} and \eqref{eq:Segre_DimX_AffineForm}, and
  the fact that we work with a single reduced point in the first factor of the
  Segre class. Now we consider the case where  Line~16 detects embedded components
  which are not contained in the singular locus of $Y_{\rm red}$. First, note
  that if $q$ is a reduced point in the singular subscheme of $Y$, then either
  $q$ is contained in an embedded component which is embedded outside of the
  singular locus of $Y_{\rm red}$ or $q$ is inside the singular locus of
  $Y_{\rm red}$. Note the loop in Line 14 starts in top dimension. Suppose that $\nu$ is the largest value such that $q$ is
  contained in $W_\nu$ (in the notation of the algorithm). If $q$ is inside the
  singular locus of $Y_{\rm red}$ then its multiplicity will always be greater
  than one in $W_\nu$ and $s(q,W_\rho)=0$ for $\rho> \nu$ since $q$ cannot be
  contained in these. If $q$ is not inside the singular locus of $Y_{\rm red}$
  then it must be inside some embedded component of $Y$. In this case $q$ will
  also be contained in some higher dimensional isolated primary component, but then
  since we know all the isolated primary components are reduced, and since multiplicity
  is computed relative to the top dimensional component, we will obtain that the
  multiplicity of $q$ is one inside this highest dimensional component in which
  it is contained by Proposition \ref{prop:RadicalTestIrreducibleCase}. Hence we
  will have correctly identified the existence of any embedded components
  outside of the singular locus of $Y_{\rm red}$. The correctness of the
  remainder of the algorithm follows from the combination of Proposition
  \ref{thm:radical_isolated_comps_reduced}, for Line 21, and Corollary
  \ref{cor:test_ideal_equal_radical}, for Lines 22--25. The algorithm terminates
  since there are finitely many components.
\end{proof}

\vspace{1em}
\begin{algorithm2e}[H]
  \SetFuncSty{textsc}
  \SetKw{RET}{{\sc return}}
  \KwIn{
    A homogeneous ideal $I=\langle f_0, \dots, f_r \rangle \subset \CC[x_0,\dots, x_n]$ defining a subscheme $Y$ of $\pp^n$.
  }
  
  \KwOut{ \textsc{true} if $Y$ is reduced (i.e.,~if $I:\langle x_0, \dots, x_n \rangle^\infty$ is radical) and \textsc{false} otherwise.}
  
  Set $\ell_0(x)$ to be a generic (random) homogeneous linear form in $\CC[x_0,\dots, x_n]$\;Set $I={I}+\langle \ell_0-1 \rangle$; Set $N=I$; Set $\hat{Y}$ to be the subscheme of $\CC^{n+1}$ associated to $N$\;
  
  \For{$\nu=\dim(\hat{Y}), \dots, 0$ }{
    Compute (at least) one generic point in each dim.~$\nu$ irreducible component of $\hat{Y}$ via the method summarized in Corollary \ref{cor:RUR}. Call the resulting collection of points $\{{p}_1,\dots , {p}_r\} \subset \CC^{n+1}$\; 
    \For{$p \in \{p_1,\dots , p_r\}$ }{
    Let $W(p)$ denote the isolated primary component of $\hat{Y}$ containing $p$\;
    Compute $\{s(p,W(p))\}_0$ using \eqref{eq:Segre_DimX_AffineForm}\;
    \If{ $\{s(p,W(p))\}_0\not= 1$}
    {\RET \textsc{false} \;}
    }
    }
    Set $c={\rm codim}(Y, \PP^n)$\;
    Set $J=I+\mathfrak{Jac}_c(I)$ where $\mathfrak{Jac}_c(I)$ is the ideal generated by the $c\times c$ minors of the Jacobian matrix of $f_0,\dots, f_r$\;
     Using the method of Corollary \ref{cor:RUR} compute (at least) one generic point in each irreducible component of the singular subscheme $\mathbb{V}(J)$ of $\hat{Y}$, call this collection of points $\{{q}_1,\dots , {q}_s\} \subset \CC^{n+1}$\;
        \For{$q \in \{q_1,\dots , q_r\}$ }{
        \For{$W_\nu$ the union of isolated primary  comp.~of $\hat{Y}$ of dim.~$\nu$, starting from top dim.~$\nu=\dim(Y)$}{
    Compute $\{s(q,W_\nu) \}_0$ using  \eqref{eq:Segre_DimX_AffineForm} \;
    \If{ $\{s(q,W_\nu) \}_0= 1$}
    {$q$ is contained in an embedded component of $Y$ but not in the singular locus of $Y_{\rm red}$\;
    \RET \textsc{false} \;}
    \If { $\{s(q,W_\nu) \}_0>0$}{
    \textsc{break} to line 13 and select the subsequent point $q$ in the list\;
    }
    }
    }
    Compute the saturation $K=I:J^\infty$ and let $Y_{\rm red}=\VV(K)$, note $K$ is equal to $\sqrt{I}$ by Proposition \ref{thm:radical_isolated_comps_reduced}\;
       \If{ $s(Y_{\rm red}, \PP^n) \not  = s(Y,\PP^n)$}
    {\RET \textsc{false} \;}
    \Else{  \RET \textsc{true} \;}

  \caption{\FuncSty{is\_Radical}}
  \label{alg:is_radical_V1}
  \label{alg:is_radical}
\end{algorithm2e}

\begin{remark}[Numerical Version of Algorithm \ref{alg:is_radical}]
\label{remark:NumericVersion}
One could implement a version of Algorithm~\ref{alg:is_radical} using
methods from numerical algebraic geometry \cite{sommese2005numerical} for point sampling and Segre class computations (via \eqref{eq:Segre_DimX_AffineForm}). For both tasks, implementations such as the stand alone packages Bertini~\cite{Bertini}, PHCPack~\cite{PHCPack}, or the {\tt NumericalAlgebraicGeometry} Macaulay2~\cite{M2} package could be employed. The only step which could not be implemented, at least straightforwardly, numerically is Line~21. However, we
do not focus on this case here. We instead focus on the case where points are represented symbolically since our primary aim is to give
complexity bounds and these are well understood for the symbolic methods such
as \cite{Rojas-sparse-rur} which we employ (via the adaptations
discussed in \S\ref{sec:mult-and-RUR} below, see Remark \ref{remark:usingRUR} specifically). To
furnish a reliable numerical implementation would require one to establish
bounds on the precision needed in the points sampled in Lines 4 and 12 to ensure
that the numeric versions of the degree computations in \eqref{eq:Segre_DimX_AffineForm} would produce the
correct results using the approximate points (since \eqref{eq:Segre_DimX_AffineForm} uses the
defining equations of the point, which would in the numeric case be
approximate).
\end{remark}

We now prove a worst case complexity bound for the operations preformed in Algorithm \ref{alg:is_radical} above. As noted in the Introduction this bound depends on the geometry of the input scheme $Y\subset \pp^n$.
    \begin{theorem}
      Let $I=\langle f_0, \dots, f_r \rangle$ be a homogeneous ideal in
      $k[x_0,\dots, x_n]$ with $\deg(f_i)\leq d$ defining a scheme $Y$ in
      $\pp^n$. If $Y$ has reduced isolated primary components and has no embedded
      components outside of the singular locus of
      $Y_{\rm red}=\mathbb{V}(\sqrt{I})$, then Algorithm~\ref{alg:is_radical}
      computes the radical via a single ideal saturation and its  worst case
      complexity is doubly exponential in $n$. In all other cases, Algorithm~\ref{alg:is_radical} has worst case complexity bounded in
      $d^{\mathcal{O}(n^4)}$, that is singly exponential in $n$.
\end{theorem}
\begin{proof}
  That the saturation is only computed if $Y$ has only reduced isolated primary
  components and has no embedded components supported outside of the singular locus of
  $Y_{\rm red}=\mathbb{V}(\sqrt{I})$ follows from the proof of Theorem
  \ref{thm:AlgCorrect}. The complexity of the saturation in this case is that of
  a Gr\"obner basis in an elimination order, i.e.,~doubly exponential in $n$,
  e.g.,~\cite{Yap-AA-2000}. In all other cases the complexity is that of the
  $\deg_\nu$ function applied to an ideal generated by polynomials of degree no
  more than $d^n$, as this is the upper bound on the degree of the generators of
  the ideal $J$ defined in line 11 of Algorithm~\ref{alg:is_radical}. The
  conclusion then follows by Proposition~\ref{prop:Deg_Low_dim_Complexity}.
\end{proof}

\begin{remark}[The minors of the Jacobian]
We should note that Algorithm~\ref{alg:is_radical} also computes, Line~11, all the $c\times c$ minors of the Jacobian.
Each minor requires a determinant computation; this costs $\OO(c^{\omega})$, where $\omega$ is the exponent of the complexity of matrix multiplication \cite{GG-mca-2013}. There, are ${r+1 \choose c}\cdot {n+1 \choose c} \leq (16 r n )^c$ minors;
  since $c \leq n$, this number is bounded by $(16 r n)^n$.
  As $r$, the number of polynomials, is part of the input
  this bound is still singly exponential in $n$. Even more, it is reasonable to assume that for all practical cases $r =  \OO(d^n)$.
 \end{remark}

\subsection{Using Points Represented in the Rational
  Univariate Representation}
\label{sec:mult-and-RUR}
In the previous subsection we assumed that all points could be represented exactly
as maximal ideals in our ambient coordinate ring. In practice on a computer we
will in fact work in the ring $\QQ[x_1,\dots, x_m]$ and will need a way to
represent each point which may appear in our algorithm in this ring. We will use
the {\em rational univariate representation (RUR)}, see Theorem~\ref{thm:RUR}.

Consider zero
dimensional ideals defining sets of points in $\CC^{m}$;
usually for our purposes below $m=n+2$.
We now consider the representation of points in RUR and how this
interacts with multiplicity computation \eqref{eq:Multiplcity_Affine_Version} and Sergre class computation \eqref{eq:Segre_DimX_AffineForm}.
To setup the context and the notation
for a 0-dimensional ideal $\mathcal{I} \subset \QQ[x_1,\dots, x_m]$
the RUR of $(\VV(\mathcal{I}))_{\rm red}$ is the ideal
$$
J=\left\langle
R(\theta),x_1-\frac{A_1(\theta)}{R'(\theta)},\dots, x_m-\frac{A_m(\theta)}{R'(\theta)}
\right\rangle \subset \QQ(\theta)[x_1,\dots, x_m],
$$
where $R,A_1,\dots, A_m \in \QQ[\theta]$ are square-free polynomials.
Equivalently, we will consider the RUR of $(\VV(\mathcal{I}))_{\rm red}$ as the  ideal
$$
\mathcal{J}=\left\langle
R(\theta),1-T\cdot R'(\theta), x_1R'(\theta)-{A_1(\theta)},\dots, x_mR'(\theta)-{A_m(\theta)}
\right\rangle \subset \QQ[x_1,\dots, x_m,\theta,T].
$$
Further, there is a $\QQ$-algebra isomorphism between $\QQ[x_1,\dots, x_m]/\sqrt{\mathcal{I}}$ and $\QQ(\theta)[x_1,\dots, x_m]/J$, and between $\QQ[x_1,\dots, x_m]/\sqrt{\mathcal{I}}$ and $\QQ[x_1,\dots, x_m,\theta,T]/\mathcal{J}$, and the variety $(\VV(\mathcal{I}))_{\rm red}$ consists of $\deg(R(\theta))$ reduced points \cite{Rou:rur:99}.

The following result let us compute the zero dimensional part of the Segre class of a set of points inside a scheme when the set of points is represented using a RUR. 
\begin{theorem}
  Let $I=\langle f_1,\dots, f_r \rangle $ be a homogeneous ideal in
  $\QQ[x_0,\dots, x_n]$ defining a scheme $Y$ in $\pp^n$ and let
  $X$ be a zero dimensional variety (i.e.,~a union of
 reduced points) contained in $Y$. Suppose
  $I_0=\langle g_0, \dots, g_\mu \rangle$ is the radical ideal defining $X$. Fix a (general) dehomogenization of
  $I$ via the linear form $\ell_0(x)$ corresponding to a generic affine patch of
  $\pp^n$; this gives the scheme $\hat{Y}$ in $\CC^{n+1}$ defined by the ideal
  $\hat{I}=\langle f_1,\dots, f_r ,\ell_0-1\rangle$ and a set of points
  $\hat{X}\subset \CC^{n+1}$ defined by the ideal
  $\hat{I}_0=I_0+\langle \ell_0-1 \rangle$.
  Let $$\mathcal{J}=\left\langle R(\theta),1-T\cdot R'(\theta), x_1R'(\theta)-{A_1(\theta)},\dots, x_nR'(\theta)-{A_n(\theta)} \right\rangle=\langle g_0,g_1,\dots, g_{n+1} \rangle$$
  be a polynomial ideal in $\QQ[x_0,\dots, x_n,\theta,T]$ giving a RUR of $\hat{X}$. Finally consider the polynomial
  ring $\QQ[x_0,\dots, x_n,\theta,t,T]$ and define the
  ideal $$ \mathcal{I}=\hat{I}+\langle P_1,\dots, P_{\dim(Y)},1-tP_0 \rangle \subset \QQ[x_0,\dots, x_n,\theta,t,T],
$$ where $P_j=\sum_{i=_0}^{\mu}\lambda_i^{(j)} g_i$ for general $\lambda_i^{(j)}\in \QQ$. Then $$
 \{s(X,Y)\}_0=\deg(Y)\cdot \deg(R(\theta))^{\dim(Y)}-\dim_\QQ\left( \QQ[x_0,\dots, x_n,\theta,t,T]/\mathcal{I}\right).
 $$\label{thm:Segre_0_comp_RUR}
\end{theorem}
\begin{proof}
  Pick a set of homogeneous generators for
  ${{I}_0}=\langle w_1,\dots w_s\rangle$ with $\deg(w_i)=\deg(R(\theta))$ and
  set $$ \mathfrak{V}=\hat{I}+\left\langle \sum_{i=_1}^{s}\rho_i^{(1)} w_i,\dots, \sum_{i=_1}^{s}\rho_i^{(\dim(Y))} w_i,1-t\left( \sum_{i=_1}^{s}\rho_i^{(0)} w_i \right)\right\rangle \subset \QQ[x_0,\dots, x_n,t],
 $$ for general $\rho_i^{(j)}\in \QQ$. Since $\deg(w_i)=\deg(R(\theta))$ then by \eqref{eq:ProjDegIdealEqs}, \eqref{eq:ProjDegIdeal}, and \eqref{eq:SegreDimX} we have that  that  $$\{s(X,Y)\}_0=\deg(Y)\cdot \deg(R(\theta))^{\dim(Y)}-\dim_\QQ\left( \QQ[x_0,\dots, x_n,t]/\mathfrak{V}\right).$$ Hence to prove the result it is sufficient to establish an equality between $\dim_\QQ\left( \QQ[x_0,\dots, x_n,t]/\mathfrak{V}\right)$ and $\dim_\QQ\left( \QQ[x_0,\dots, x_n,\theta,t,T]/\mathcal{I}\right)$.
 
 By the definition of the RUR  we have a $\QQ$-algebra isomorphism, which we will denote $\Phi$, between $\QQ[x_0,\dots, x_n,t]/\langle w_1,\dots, w_s\rangle$ and $\QQ[x_0,\dots, x_n,t,\theta,T]/\mathcal{J}$, c.f.~\cite{Rou:rur:99,AlBeRoWo-idem-96}. This isomorphism induces a map, which we also call $\Phi$, between the quotient rings $\QQ[x_0,\dots, x_n,t]/\hat{I}$ and $\QQ[x_0,\dots, x_n,t,\theta,T]/\hat{I}$. Necessarily this induced map $$\Phi:\QQ[x_0,\dots, x_n,t]/\hat{I}\to \QQ[x_0,\dots, x_n,t,\theta,T]/\hat{I}$$ must map a general polynomial of a fixed degree in $\langle w_1,\dots, w_s\rangle$  to a general polynomial of the same degree in $\mathcal{J}$; in particular we have that $$
 \Phi\left(\sum_{i=_1}^{s}\rho_i w_i\right)=\sum_{i=_0}^{n+1}\lambda_i g_i,
 $$
 for some general constants $\lambda_i$ which will be determined by the general $\rho_i$. Since dimensions of the $\QQ$-vector spaces being considered do not depend on the choice of constants, provided they are general, it follows that 
 $$\dim_\QQ\left( \QQ[x_0,\dots, x_n,t]/\mathfrak{V}\right)=\dim_\QQ(\Phi( \QQ[x_0,\dots, x_n,t]/\mathfrak{V}))=\dim_\QQ\left( \QQ[x_0,\dots, x_n,\theta,t,T]/\mathcal{I}\right).$$
 In other words, the image of $\Phi$ applied to the quotient ring $\QQ[x_0,\dots, x_n,t]/\mathfrak{V}$ may differ from the quotient ring $\QQ[x_0,\dots, x_n,\theta,t,T]/\mathcal{I}$, but only up to a choice of general constants, which does not effect the resulting vector space dimension. 
\end{proof}

This result in conjunction with Proposition~\ref{prop:SegreTest} below, which can be used used in Lines 8 and 16 of Algorithm \ref{alg:is_radical_V1} via the RUR  of a zero dimensional variety  $\hat{X}$. As noted
in the proof this criterion is in effect testing all the multiplicities of dimension zero components at once.

\begin{proposition}
  Let $Y$ be a scheme in $\pp^n$, let $W$ be the union of all $\nu$-dimensional primary components of $Y$ and let $X$ a dimension zero variety fully contained in $W$ and such that no point in $X$ is in the singular locus of $W_{\rm red}$.  Then $\{s(X,W)\}_0>\deg(X)$ if and only if $W$ is not reduced.  \label{prop:SegreTest}
\end{proposition}
\begin{proof}
 By \cite[Example~4.3.4]{Fulton} we know that $\{s(X,W)\}_0$, the coefficient of the dimension $0$ part of the Segre class, is the sum of the multiplicities of $W$ along each point in $X$. For $W$ to be reduced we require that it is reduced at a generic point (i.e.,~one outside the singular locus of $W_{\rm red}$), in light of Proposition \ref{prop:RadicalTestIrreducibleCase} we see that in particular we need $e_pW=1$ for each $p\in X$ (since $\dim(X)=0$), hence taken together, for $W$ to be reduced we require the integer $\{s(X,W)\}_0$ to be equal the number of points in $X$. Conversely if $W$ is reduced and all points in $X$ are smooth in $W_{\rm red}$ (again via Proposition \ref{prop:RadicalTestIrreducibleCase}) $e_pW=1$, for all $p\in X$; it follows that in this case we necessarily have $\{s(X,W)\}_0=\deg(X)$.
\end{proof}

\begin{remark}[Using the RUR in Algorithm \ref{alg:is_radical}] To use the RUR of sampled points in Algorithm \ref{alg:is_radical} we make the following alterations:
\begin{itemize}
    \item $p$ in Line 6 and $q$ in Line 13 are now both zero-dimensional varieties equal to the union of some set of reduced points represented as a RUR;
    \item in Line 7 we use Theorem \ref{thm:Segre_0_comp_RUR} in conjunction with \eqref{eq:Segre_DimX_AffineForm} to obtain the Segre class and the criterion in Line 8 becomes $\{s(p,W(p))\}_0> \deg(p)$ by Proposition \ref{prop:SegreTest};
        \item in Line 15 we use Theorem \ref{thm:Segre_0_comp_RUR} in conjunction with \eqref{eq:Segre_DimX_AffineForm} to obtain the Segre class and the criterion in Line 16 becomes $\{s(q,W_\nu)\}_0 = \deg(q)$ by Proposition \ref{prop:SegreTest}.
\end{itemize}
The rest of the algorithm remains as presented in Algorithm \ref{alg:is_radical}.\label{remark:usingRUR}
\end{remark}

\paragraph{Acknowledgements}

The authors are grateful to Peter B\"urgisser for various discussions and suggestions on the problem of computing the radical.
ET is  partially supported by ANR JCJC GALOP (ANR-17-CE40-0009) and the PHC GRAPE.

\small
\bibliographystyle{abbrv}
\bibliography{segre}

\end{document}